\documentclass[11pt]{article}
\usepackage{amsmath, amsthm, amscd, amssymb,tikz,cite,graphicx}
\usepackage{caption,subcaption}
\usepackage{thm-restate,hyperref,relsize,enumerate}

\numberwithin{equation}{section}

\usepackage{color}

\theoremstyle{plain}
\newtheorem{theorem}{Theorem}[section]
\newtheorem*{cor-acyclic}{Corollary~\ref{cor:acyclic}}

\newtheorem{proposition}[theorem]{Proposition}
\newtheorem{lemma}[theorem]{Lemma}

\newtheorem{claim}[theorem]{Claim}

\theoremstyle{definition}

\theoremstyle{plain}

\begin{document}

\title{Decomposing planar graphs into graphs with degree restrictions}

\author{\small Eun-Kyung Cho\thanks{
Department of Mathematics, Hankuk University of Foreign Studies, Yongin-si, Gyeonggi-do, Republic of Korea. \texttt{ekcho2020@gmail.com}
},  \  \
\small Ilkyoo Choi\thanks{
Department of Mathematics, Hankuk University of Foreign Studies, Yongin-si, Gyeonggi-do, Republic of Korea.
\texttt{ilkyoo@hufs.ac.kr}
},  \  \
\small Ringi Kim\thanks{
Department of Mathematical Sciences, KAIST, Daejeon, Republic of Korea.
\texttt{kimrg@kaist.ac.kr}},  \  \
 \small Boram Park\thanks{
Department of Mathematics, Ajou University, Suwon-si, Gyeonggi-do, Republic of Korea.
\texttt{borampark@ajou.ac.kr}}
,  \  \
\small Tingting Shan\thanks{
Department of Mathematics, Zhejiang Normal University, China.
\texttt{15735291101@163.com}
},  \  \
\small Xuding Zhu\thanks{
Department of Mathematics, Zhejiang Normal University, China.
\texttt{xdzhu@zjnu.edu.cn}
}}
\date\today
\maketitle
\begin{abstract}
Given a graph $G$, a   decomposition of $G$ is a partition of its edges.
A graph is   $(d, h)$-decomposable if its edge set can be partitioned into a $d$-degenerate graph and a graph with maximum degree at most $h$. For $d \le 4$, we are interested in the minimum integer $h_d$ such that 
every planar graph is $(d,h_d)$-decomposable.  
It was known that $h_3 \le 4$ and $h_2\le 8$ and $h_1 = \infty$. This paper proves that $h_4=1, h_3=2$ and $4 \le h_2 \le 6$.   
\end{abstract}

\section{Introduction}

We consider only finite simple graphs.
Given a graph $G$, a {\it decomposition} of $G$ is a collection of spanning subgraphs $H_1,\ldots, H_t$ such that each edge of $G$ is an edge of $H_i$ for exactly one $i\in\{1, \ldots, t\}$.
In other words, $E(H_1), \ldots,E(H_t)$ is a partition of $E(G)$.

A graph is {\it $d$-degenerate} if every subgraph has a vertex of degree at most $d$.
Given non-negative integers $d$ and $ h$, a {\it $(d,h)$-decomposition} of a graph $G$ is a decomposition $H_1, H_2$ of $G$ such that $H_1$ is $d$-degenerate and $H_2$ has maximum degree at most $h$. We say $G$ is 
{\it $(d,h)$-decomposable} if there exists a $(d,h)$-decomposition of $G$. This paper studies $(d,h)$-decomposability of planar graphs. 

Decomposing a  graph into subgraphs with simpler structure is a fundamental problem 
in graph theory. The classical Nash-Williams Arboricity Theorem~\cite{nash1964decomposition} (see also~\cite{nash1961edge,tutte1961problem}) gives a necessary and sufficient condition under which a graph can  be decomposed into $k$ forests. The Nine-Dragon Tree Conjecture \cite{montassier2012decomposing}, confirmed by Jiang and Yang \cite{jiang2017decomposing}, gives a sharp density condition under which a graph can be decomposed into $k$ forests
with one of them having bounded maximum degree. The page number of a graph $G$ is the minimum $k$ such that $G$ can be decomposed into $k$ planar graphs. A proper edge colouring of $G$ is a decomposition of $G$ into matchings.  The problem of decomposing a graph $G$ into star forests, linear forests, graphs of bounded maximum degree, etc., are studied extensively in the literature. 

The concept of $(d,h)$-decomposition has not been 
  defined formally in the literature (as to our knowledge). 
  However, such decompositions raise naturally in the study of 
  many  problems.  
  For example, it was proved in \cite{guan1999game} that if a graph $G$ is $(1,h)$-decomposable, then $G$ 
has game chromatic number   $\chi_g(G) $ at most $  4+h$. It was shown in \cite{guan1999game}
that outerplanar graphs are  $(1,3)$-decomposable, and hence have game chromatic number at most $7$. 
A  result in \cite{2000Zhu} implies that planar graphs are $(2,8)$-decomposable, and such a decomposition (with some more structure constraints) was used to show that planar graphs have game chromatic number at most $19$ (the currently best known upper bound for the game chromatic number of planar graphs is $17$ \cite{zhu2008refined}). A similar decomposition  were used to derive upper bound on the game chromatic number of graphs $G$   embeddable on an orientable surface of genus $g \ge 1$, namely, $\chi_g(G) \le \lfloor \frac 12(3\sqrt{1+48g)} +23) \rfloor$.  
It is known that
if $G$ decomposes into $H_1, H_2, \ldots, H_k$, then the spectral radius   of $G$ is bounded by the summation of the spectral radius of $H_i$, i.e., $\rho(G) \le \rho(H_1)+\rho(H_2)+\cdots + \rho(H_k)$ \cite{weyl1912asymptotische, dvovrak2010spectral} .
The currently best known upper bounds on the spectral radius of planar graphs (namely, $\rho(G) \le \sqrt{8\Delta - 16}+3.47$)   was obtained by Dvo\v{r}\'{a}k and Mohar
 \cite{dvovrak2010spectral} by applying the result  that every planar graph $G$ decomposes into $H_1, H_2$, with 
 $H_1$ has an orientation of maximum out-degree $2$, and $H_2$ has maximum degree at most $4$.

In this paper,   we are interested in the minimum integer $h_d$ such that 
every planar graph is $(d,h_d)$-decomposable. Since every planar graph is $5$-degenerate, the problem is interesting only for $d \le 4$. As observed above, a result in \cite{2000Zhu} implies that every planar graph is $(2,8)$-decomposable, i.e., $h_2 \le 8$.  A result in~\cite{gonccalves2009covering}  implies that every planar graph is $(3, 4)$-decomposable, i.e., $h_3 \le 4$. In this paper, we prove the following results:

\begin{theorem}\label{thm:4DplusM}
Every planar graph is $(4,1)$-decomposable.  
\end{theorem}

\begin{theorem}\label{thm:3Dplus2}
Every planar graph is $(3,2)$-decomposablee.
\end{theorem}

Since planar graphs of minimum degree $5$  is neither $(3,1)$-decomposble nor $(4,0)$-decomposable, 
we conclude that $h_4=1$ and $h_3=2$.

\begin{theorem}\label{thm:2Dplus6}
Every planar graph is $(2,6)$-decomposable.  
\end{theorem}

\begin{proposition}\label{prop:(2,3)}
Not all planar graphs are $(2,3)$-decomposable.
\end{proposition}

As a consequence of Theorem \ref{thm:2Dplus6} and Proposition \ref{prop:(2,3)}, we have $4 \le h_2 \le 6$. The exact value of $h_2$ remains an open problem. 

Note that for every integer $h$,  the complete bipartite graph with two vertices in one part and $2h+2$ vertices in the other part  is not $(1, h)$-decomposable. Thus $h_1 = \infty$.  

A graph $G$ is {\it $h$-defective $k$-choosable} if for any $k$-list assignment $L$ of $G$, there is an $L$-colouring of $G$ in which each vertex $v$ has at most $h$-neighbours coloured the same colour as $v$.
The concept of {\it $h$-defective $k$-paintable} is an online version of $h$-defective $k$-choosable, defined through
a two-person game (see \cite{gutowski2018defective} for its definition), and {\em $h$-defective $k$-DP-colourable} is a generalization of 
$h$-defective $k$-choosable (see \cite{jing2019defective} for its definition). 
We remark that $(d,h)$-decomposable graphs are easily seen to be  $h$-defective $(d+1)$-choosable,   $h$-defective $(d+1)$-paintable, as well as  $h$-defective $(d+1)$-DP-colourable.  
On the other hand, $(d,h)$-decomposable seems to be considerably  stronger than $h$-defective $(d+1)$-choosability and $h$-defective $(d+1)$-paintability. Cushing and Kierstead~\cite{2010CuKi}   proved that every planar graph is $1$-defective $4$-choosable. This result was strengthened recently by  Grytczuk and Zhu~\cite{2020GrZh}   who proved that every planar graph is $1$-defective $4$-paintable. As observed above, planar graphs with minimum degree $5$ are not $(3,1)$-decomposable.  Eaton and Hull~\cite{1999EaHu}, and independently \v Skrekovski~\cite{1999Sk} proved that every planar graph is $2$-defective $3$-choosable. Gutowski, Han, Krawcyzk and Zhu \cite{GHKZ}
[ Defective 3-paintability of planar graphs, Electronic Journal of Combinatorics, Volume 25, Issue 2 (2018) ] showed that there are planar graphs that are not $2$-defective $3$-paintable, but every planar graph is $3$-defective $3$-paintable. 
We show in this paper that not every planar graph is $(2,3)$-decomposable.

The proof of Theorem~\ref{thm:4DplusM} (given in Section~\ref{sec:4DplusM}) uses standard discharging method.     Theorem~\ref{thm:2Dplus6} and Theorem~\ref{thm:3Dplus2}  are obtained  by proving  stronger and more 
technical statements in Section~\ref{sec:2Dplus6} and Section~\ref{sec:3Dplus2}, respectively.
The technical statement used to derive Theorem \ref{thm:2Dplus6} is more intriguing and the proof is also
more complicated. 

We end this section with some definitions and notation.
A vertex ordering $\sigma$ of $G$ is \emph{$d$-degenerate} if every vertex has at most $d$ earlier neighbors in the ordering $\sigma$.
Note that a graph $G$ is $d$-degenerate if and only if it has a $d$-degenerate ordering.
For $S \subseteq V(G)$ and a vertex ordering $\sigma$ of $G$, let $\sigma -S$ denote a subordering of $\sigma$ obtained by deleting the vertices in $S$.
We also note that a graph $G$ is $d$-degenerate if and only if it has an acyclic orientation whose maximum out-degree is at most $d$. Therefore, when we prove Theorems~\ref{thm:3Dplus2}~and~\ref{thm:2Dplus6}, we find a pair $(D,H)$, where $H$ is a subgraph of $G$ with  $\Delta(H)\le h$ and $D$ is  an acyclic orientation of $G-E(H)$  with
$\Delta^+(D)\le d$.
 
 Let $G$ be a plane graph.
A \emph{plane subgraph} of $G$ is a subgraph of $G$ whose plane embedding is inherited.
We say $G$ is a \emph{near triangulation} if $G$ is a $2$-connected plane graph
and every face of $G$ except the outer face is a triangle.
Note that the outer face of a near plane triangulation $G$ is a cycle since $G$ is $2$-connected.
A {\it boundary vertex} and {\it boundary edge} of $G$ are a vertex and an edge, respectively, on the boundary cycle of $G$.  For a boundary edge $uv$, $v$ is called a  \emph{boundary neighbor} of $u$.

An arc, which is a directed edge, is represented by an ordered pair of vertices;
namely, $uv$ is an (undirected) edge whereas $(u,v)$ is an arc from $u$ to $v$.
For a graph $G$ and a set $E$ of unordered pairs on $V(G)$, let $G+E$ (resp. $G-E$) denote the graph obtained from $G$ by adding (resp. deleting) the elements of $E$ to (resp. from) the edge set of $G$.
If $|E|=1$, say $E=\{ww'\}$, then denote $G+E$ (resp. $G-E$) by $G+ww'$ (resp. $G-ww'$).
For a digraph $D$ and a set $A$ of ordered pairs on $V(D)$, define $D+A$, $D-A$, $D+(w,w')$, and $D-(w,w')$ similarly.
Moreover, for a digraph $D$  and vertices $x,y \in V(D)$, let $D-xy$ denote the subdigraph $D-\{(x,y),(y,x)\}$.
We often drop the parentheses to improve the readability.
For instance, for a digraph $D$ and sets $A_1$, $A_2$, $A_3$ of ordered pairs on $V(D)$,
both $D-A_1+A_2+A_3$ and $D-A_1+ (A_2 + A_3)$ denote $((D-A_1)+A_2)+A_3$.

For two (di)graphs $G_1$ and $G_2$,  let $G_1\cup G_2$ be the (di)graph such that $V(G_1\cup G_2)=V(G_1)\cup V(G_2)$ and $E(G_1\cup G_2)=E(G_1)\cup E(G_2)$.

\section{Proof of  (2,6)-decomposability}\label{sec:2Dplus6}

Assume $G$ is a near triangulation, $xy$ is a  boundary edge of $G$,  and $w$ is a boundary vertex  of  $G$. We denote by $b_{G,xy}(w)$   the number of vertices in $\{x,y\}$ that are boundary neighbors of $w$. Recall that $w$ is a boundary neighbor of $x$ if $xw$ is a boundary edge of $G$.
If there is no confusion, then we use $b(w)$ to denote $b_{G,xy}(w)$. Instead of proving Theorem~\ref{thm:2Dplus6}
directly, we prove the following more technical result, which is easily seen to imply Theorem~\ref{thm:2Dplus6}.

\begin{theorem}\label{thm:2Dplus6:main}
Let $G$ be a near triangulation, $xy$ be a boundary edge of $G$, and $z$ be a boundary vertex of $G$ other than $x$ and $y$.
Then there exist a subgraph $H$ and an acyclic orientation $D$ of  $G-E(H)$ satisfying the following:
\begin{itemize}
\item[\rm (i)] For every interior vertex $w$, $\deg_D^+(w)\le 2$ and $\deg_H(w)\le 6$.
\item[\rm (ii)] For every boundary vertex $w$, $\deg_D^+(w)\le 1$ and $\deg_H(w)\le 5-b(w)$.
\item[\rm{(iii)}] $\deg_D^+(y)=\deg_H(y)=0$, $N_D^+(x)=\{y\}$, and $\deg_H(x)\le 1$.
If $\deg_H(x)=1$, then the neighbor $s$ of $x$ in $H$ is a boundary vertex and $s\in N_G(x)\cap N_G(y)$.
\item[\rm(iv)]  $\deg_H(z)\le 4-b(z)$.
If  equality holds, then $\deg_H(w)\le 4-b(w)$ for every boundary neighbor $w$ of $z$.
\item[\rm(v)]  For the
boundary neighbors
 $z'$ and $z''$ of $z$,   $\deg_H(z)+\deg_{H}(z')+\deg_{H}(z'') \le 12-b(z')-b(z'')$.
\end{itemize}
Let us call such $(D,H)$ a $(2,6)$-decomposition of $G$ with respect to $(x,y,z)$.
\end{theorem}

\begin{lemma}\label{obs_realization}
Let $G$ be a near triangulation, $xy$ be a boundary edge of $G$, and $z$ be a boundary vertex of $G$ other than $x$ and $y$.
If $(D,H)$ is a $(2,6)$-decomposition of $G$ with respect to $(x,y,z)$, then there is a $(2,6)$-decomposition of  $G$ with respect to $(y,x,z)$.
\end{lemma}

\begin{proof} Let $(D,H)$ be a $(2,6)$-decomposition of $G$ with respect to $(x,y,z)$.
If $\deg_H(x)=0$, then let $D'=D-(x,y)+(y,x)$ and $H'=H$.
If $\deg_H(x)=1$, then let $w$ be the neighbor of $x$ in $H$.
Then $w$ is a boundary vertex   and $(w,y)$ is an arc of $D$.
Let $D'=D-(x,y) - (w, y)+(w,x)+(y,x)$
and $H'=(H+wy)-wx$.
Then  $(D',H')$ is a $(2,6)$-decomposition of $G$ with respect to $(y,x,z)$.
\end{proof}

\begin{proof}[Proof of Theorem~\ref{thm:2Dplus6:main}]
We use induction on $|V(G)|$. If $|V(G)|= 3$, then $G=K_3$.
Let $D$ be a digraph with two arcs $(x,y)$ and  $(z,y)$,  and $H$ be a graph with one edge $xz$.  
Then $(D,H)$ is a $(2,6)$-decomposition of $G$ with respect to $(x,y,z)$. 
Suppose $|V(G)|\ge 4$.
Let $C$ be the boundary cycle of $G$, and let $z'$ and $z''$ be the boundary neighbors of $z$.
For simplicity, we denote $b_{G,xy}(w)$ by $b(w)$.

\medskip

\noindent {\bf Case 1}  $C=(x,y,z)$ is a triangle.

Let $G'=G-z$. Since $G$ contains at least four vertices, $G'$ is a near triangulation.
Let $C'$ be the boundary cycle of $G'$, and let $w$ be a boundary vertex of $G'$ other than $x$ and $y$.
See Figure~\ref{fig:(2,6)_triangle}.
By the induction hypothesis,
there is a $(2,6)$-decomposition $(D',H')$ of $G'$ with respect to $(x,y,w)$.

\begin{figure}[h!]
  \centering
  \includegraphics[width=3.2cm,page=8]{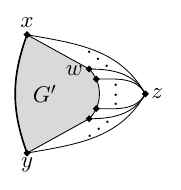}\\
  \caption{An illustration for \textbf{Case 1}}
  \label{fig:(2,6)_triangle}
\end{figure}

If $\deg_{H'}(x)=0$, then let $D=D'+\{(u,z)\mid u \in V(C')\setminus \{x,y\}\}+(z,y)$ and $H=H'+xz$.
If $\deg_{H'}(x)=1$, then for the vertex $s$ with $sx\in E(H')$,
$s$ belongs to $N_{G'}(x)\cap N_{G'}(y)\cap V(C')$ by Condition (iii), so let  $D=D'+\{(s,x),(z,y)\}+\{(u,z)\mid u \in V(C')\setminus\{x,y,s\}\}$ and $H=(H'-sx)+\{sz,xz\}$.

In both cases, we can easily check Conditions (i)-(iii).
Since $b(z)=2$ and $\deg_H(z)\le 2$, Condition (iv) holds.
Since $b(z')+b(z'')=2$, we have
$\deg_H(z)+\deg_H(z')+\deg_H(z'')=\deg_H(z)+\deg_H(x)+\deg_H(y)\le 2+1=3\le 12-2$, so Condition (v) holds.
Thus $(D,H)$ is a $(2,6)$-decomposition of $G$ with respect to $(x,y,z)$.

\medskip

\noindent {\bf Case 2} $C$ has  a chord $uv$ that either separates $xy$ and $z$ or is incident with one of $x,y,z$.

Let $G_1$ and $G_2$ be the plane subgraphs of $G$ separated by $uv$. 
Namely, $G_1=G[V_1], G_2=G[V_2]$, where $V_1 \cup V_2 =V(G)$ and $V_1 \cap V_2=\{u,v\}$.  Then each $G_i$ is a near triangulation. Let $C_i$ be the boundary cycle of $G_i$. 
Without loss of generality, assume $x,y \in V(G_1)$.
 We divide the proof into three subcases: (1) $z\not\in V(G_1)$, (2) $z\in \{u,v\}$, and (3) $z\in V(G_1)\setminus \{u,v\}$.
In each case,
we will find a $(2,6)$-decomposition of $G_1$ with respect to $(x,y,w')$ for some $w'\in\{z,u,v\}$, and a $(2,6)$-decomposition $(D_2,H_2)$ of $G_2$ with respect to $(u,v,w^*)$ or $(v,u,w^*)$ for some vertex $w^*$.
Let $D=D_1 \cup (D_2-uv)$ and $H=H_1 \cup H_2$.
It is clear that $D$ is acyclic.

For simplicity, denote $b_{G_1,xy}(w)$ and $b_{G_2,uv}(w)$ by $b_1(w)$ and $b_2(w)$, respectively.
If $w\in V(G_i)\setminus\{u,v\}$, then
$\deg^+_{D}(w)=\deg^+_{D_i}(w)$,
$\deg_{H}(w)=\deg_{H_i}(w)$. If $w$ is a boundary vertex of $G$, then $b_{i}(w)\ge b(w)$.
If $w\in \{u,v\}$, then $\deg^+_{D}(w)=\deg^+_{D_1}(w)$,
$\deg_{H}(w)=\deg_{H_1}(w)+\deg_{H_2}(w)$, and $b_1(w)\ge b(w)$.
Hence, Condition (i) immediately  holds, and Conditions (ii)-(v) hold
 except those regarding the degrees in $H$ involving $u$ or $v$.
From now on, we will prove that Condition (ii) holds  when $w \in \{u,v\}$, Condition (iii) holds when $x$ or $y$ is in $\{u,v\}$, and Conditions (iv) and (v) hold when $z$, $z'$, or $z''$ is in $\{u,v\}$.

\begin{figure}[h!]
  \centering
  \includegraphics[width=4.5cm,page=2]{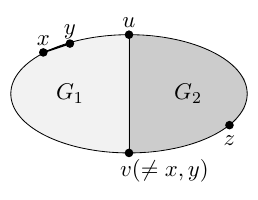} \quad
    \includegraphics[width=4.5cm,page=3]{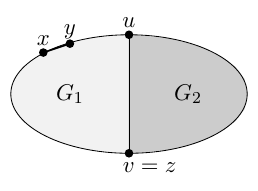}
  \quad
    \includegraphics[width=4.5cm,page=4]{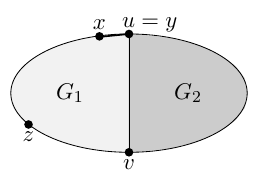}\\
  \caption{Illustrations for \textbf{Case 2}}
    \label{fig:(2,6)_z_not_in_G_1}
\end{figure}

\smallskip

\noindent {\bf Case 2-1}  $z\not\in V(G_1)$.

We may assume $v\not\in\{ x,y\}$. See the first figure of Figure~\ref{fig:(2,6)_z_not_in_G_1}.
By Lemma~\ref{obs_realization}, we may assume $x\neq u$.
Note that $z\not\in\{u,v\}$ and  $z', z'' \in V(C_2)$.
Let $(D_1,H_1)$ be a (2,6)-decomposition of $G_1$ with respect to $(x,y,v)$.
If $\deg_{H_1}(u)\le 4-b_1(u)$ and $y\neq u$, then let $(D_2,H_2)$ be a $(2,6)$-decomposition of $G_2$ with respect to $(u,v,z)$.  Otherwise, let $(D_2,H_2)$ be a $(2,6)$-decomposition of $G_2$ with respect to $(v,u,z)$. Recall that it is enough to check Conditions (ii)-(v) for the case where  $\deg_{H}(u)$ or $\deg_H(v)$ is involved.

Condition (ii) holds, since
\begin{eqnarray*}
&&\deg_H(v)\le \deg_{H_1}(v)+\deg_{H_2}(v) \le (4-b_1(v))+1 \le 5-b(v)\\
&&\deg_H(u)\le \deg_{H_1}(u)+\deg_{H_2}(u) \le
\begin{cases}
(4-b_1(u))+1 \le 5-b(u) &\text{if }\deg_{H_1}(u)\le 4-b_1(u)\\
(5-b_1(u))+0 \le 5-b(u) &\text{otherwise.}
\end{cases}
\end{eqnarray*}
If $u=y$, then $\deg_{H}(y)=\deg_{H_1}(y)+\deg_{H_2}(y)=0+0=0$, which implies Condition (iii).
To check Condition (iv), suppose that $\deg_H(z)=4-b(z)$.
Since $\deg_H(z)=\deg_{H_2}(z)\le 4-b_2(z)\le 4-b(z)$, we conclude that  $b_2(z)=b(z)$ and  $\deg_{H_2}(z)= 4-b_2(z)$.
Since $b_2(z)=b(z)$,  either $b(z)=b_2(z)=0$ or $u=y$ and $y$ is a boundary neighbor of $z$. For the first case, $\{z',z''\} \cap \{u,v\}=\emptyset$, so $u$ and $v$ are not involved. For the second case,
since $b_1(v)=b(v)+1$, we have the following:
\begin{eqnarray*}
\deg_H(z')\le \deg_{H_1}(z')+\deg_{H_2}(z')\le\begin{cases}
 0+0 \le 4-b(z') &\text{if }z'=u\\
 4-b_1(z') + 1 \le 4-b(z')&\text{if }z'=v
.\end{cases}
\end{eqnarray*}
Thus Condition (iv) holds.

By Condition (ii),
$\deg_{H}(z')\le 5-b(z')$ and $\deg_{H}(z'')\le 5-b(z'')$.
If $\{z',z''\}=\{u,v\}$, then  $\deg_{H_2}(z)\le 4-b_2(z)\le 2$, so
$\deg_H(z)+\deg_{H}(z')+\deg_{H}(z'')
\le 2+5-b(z')+5-b(z'')=12-b(z')-b(z'')$.
If $z'\in \{u,v\}$ and $z''\not\in\{u,v\}$, then $\deg_{H_2}(z) \le 4-b_2(z)\le 3$,
and so
\begin{eqnarray*}
&&\deg_H(z)+\deg_{H}(z')+\deg_{H}(z'')
\le 3 +5-b(z')+ \deg_{H_2}(z'') \le
12-b(z') -b(z''),
\end{eqnarray*}
where the last inequality is from Condition (iv) for $(D_2,H_2)$ stating that
$\deg_{H_2}(z) =3$ implies
$\deg_{H_2}(z'')\le 4-b_2(z'')\le 4-b(z'')$.
Therefore Condition (v) holds.

\smallskip

\noindent {\bf Case 2-2}   $z\in\{u,v\}$.

We may assume $v=z$. Let $z' \in V(G_1)$ and $z'' \in V(G_2)$. See the second figure of Figure~\ref{fig:(2,6)_z_not_in_G_1}.
If $u\in \{x,y\}$, then we may assume $u=y$ by Lemma~\ref{obs_realization}.
Now let $(D_1,H_1)$ be a $(2,6)$-decomposition of $G_1$ with respect to $(x,y,z)$.
If $\deg_{H_1}(u)\le 4-b_1(u)$ and $u\neq y$, then let $(D_2,H_2)$ be a (2,6)-decomposition of $G_2$ with respect to $(u,v,z'')$.
Otherwise, let $(D_2,H_2)$ be a $(2,6)$-decomposition of $G_2$ with respect to $(v,u,z'')$.
Conditions (ii) and (iii) hold by the same reasoning as in Case 2-1.
Moreover, since $\deg_{H}(z'')\le 4-b_2(z'')=4-b(z'')$, Condition (v) immediately follows from Condition (iv). Hence, it is enough to show that one of the following holds:
 \begin{itemize}
 	\item $\deg_H(z)\le 3-b(z)$.
 	\item  $\deg_H(z)= 4-b(z)$ and $\deg_{H}(z') \le 4-b(z')$.
 \end{itemize}

If $\deg_{H_1}(z)= 4-b_1(z)$, then by Condition (iv) for $(D_1, H_1)$,
we know $\deg_{H_1}(u) \le 4-b_1(u)$. Hence by definition,
 $(D_2, H_2)$ is a $(2,6)$-decomposition of
$G_2$ with respect to $(u,v,z'')$. Therefore $\deg_{H_2}(z)=0$ and
$\deg_{H}(z)=\deg_{H_1}(z)=4-b_1(z)+0 \le 4-b(z)$.  Moreover, if
$\deg_{H}(z)=  4-b(z)$, then $\deg_{H_1}(z)=4-b_1(z)$, and hence
by Condition (iv) for $(D_1,H_1)$, we have
$\deg_{H_1}(z')\le 4-b_1(z')$, and hence $\deg_{H}(z') = \deg_{H_1}(z') \le 4-b(z')$.

Assume  $\deg_{H_1}(z)\le 3-b_1(z)$. If $y=u$, then $b(z)=b_1(z)-1$ and $\deg_H(z) \leq \deg_{H_1}(z)+1$.
Hence  $\deg_{H}(z)\le 3-b(z)$, and we are done.

If $y \ne u$, then $b(z)=b_1(z)$ and $\deg_H(z) \leq \deg_{H_1}(z)+1$. Hence  $\deg_{H}(z)\le 4-b(z)$ holds.
Suppose to the contrary that none of two above conditions previously mentioned holds, i.e., $\deg_{H}(z) = 4 -b(z)$ and  $\deg_H(z') = 5-b(z')$.
Then $\deg_{H_1}(z)= 3-b_1(z)$ and $\deg_{H_2}(z)= 1$.
So $(D_2, H_2)$ is a $(2,6)$-decomposition with respect to $(v,u, z'')$.
By the
 definition of $(D_2,H_2)$, this implies that  $\deg_{H_1}(u)=5-b_1(u)$.
 Moreover,  $\deg_{H_1}(z') = 5-b_1(z')$ and $z' \ne x$. 
 Since $z' \ne x$ and $y \ne u$, this implies 
 $b_1(z)=0$. However,
  \[
\deg_{H_1}(z)+\deg_{H_1}(z')+\deg_{H_1}(u) = (3-b_1(z)) + (5- b_1(z') ) + (5-b_1(u)) = 13-b_1(z')-b_1(u).\]
This is a contradiction to the assumption that $(D_1,H_1)$ is a $(2,6)$-decomposition of $G_1$ with respect to $(x,y,z)$, as
 Condition (v) for $(D_1,H_1)$ is not satisfied.

\smallskip

\noindent {\bf Case 2-3}
$z\in V(G_1)\setminus \{u,v\}$.

By the case assumption, the chord $uv$ is incident with either $x$ or $y$.
By Lemma~\ref{obs_realization}, we may assume $y=u$. See the last figure of Figure~\ref{fig:(2,6)_z_not_in_G_1}. Note that $z', z'' \in V(C_1)$. Let $(D_1,H_1)$ be a $(2,6)$-decomposition of $G_1$ with respect to $(x,y,z)$, and let $(D_2,H_2)$ be a $(2,6)$-decomposition with respect to $(v,u,z^*)$, where $z^*\in V(C_2)\setminus\{u,v\}$.
Note that Condition (iii) clearly holds by definition.
Conditions (ii), (iv), (v) hold since $b_1(v)=b(v)+1$ and $\deg_H(v)\le \deg_{H_1}(v)+1$. 

\medskip

\noindent {\bf Case 3}   Neither Case 1 nor Case 2 applies, in other words,
$C$ has at least four vertices and for every chord $uv$ of $C$, the vertices $x,y,z$ lie in the same component of $G-\{u,v\}$.

\begin{figure}[h!]
  \centering
  \includegraphics[width=4.5cm,page=5]{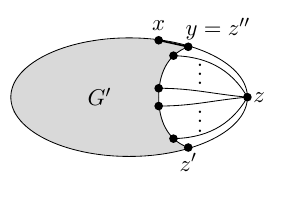}  \quad  \includegraphics[width=4.5cm,page=6]{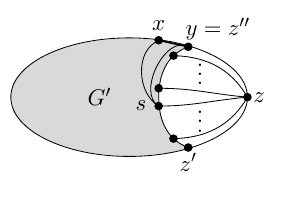}\quad   \includegraphics[width=4.5cm,page=7]{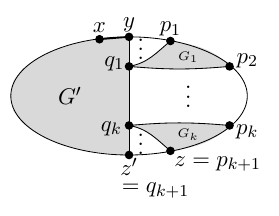}\\
  \caption{Illustrations for \textbf{Case 3}}
  \label{fig:Case3}\end{figure}

\smallskip

\noindent{\bf Case 3-1}  $z$ is a boundary neighbor of either $x$ or $y$.

By Lemma~\ref{obs_realization}, we may assume $yz \in E(C)$.
Since $C$ has at least four vertices, we may assume $z'\not\in\{x,y\}$. See the first figure of Figure~\ref{fig:Case3}.
Let $G'=G-z$, and let $P$ be the boundary path of $G'$ from $y$ to $z'$ not containing $x$.
Let $(D',H')$ be a $(2,6)$-decomposition of $G'$ with respect to $(x,y,z')$.
For simplicity, let $X=V(P)\setminus\{y,z'\}$.

If $\deg_{H'}(x)=0$, then let $D=D'+(z,y)+\{(u,z)\mid u \in X \}$ and $H=H'+zz'$.
It is easy to observe that $(D,H)$ is a $(2,6)$-decomposition of $G$ with respect to $(x,y,z)$. Suppose $\deg_{H'}(x)=1$. Then by Condition (iii) for $(D',H')$,
for the vertex $s$ with $xs \in E(H')$, $s$ belongs to $N_{G'}(x)\cap N_{G'}(y)\cap V(C')$.
See the second figure of Figure~\ref{fig:Case3}. Since $G$ has no chord incident with either $x$ or $y$, $s\in X$.
Let $D=D'+(s,x)+(z,y)+\{(u,z)\mid u \in V(P)\setminus \{y,z',s\}\}$ and $H=(H'-sx)+\{zz',sz\}$.
Then $(D,H)$ is a $(2,6)$-decomposition of $G$ with respect to $(x,y,z)$.

\smallskip

\noindent {\bf Case 3-2}
Neither $x$ nor $y$ is a boundary neighbor of $z$.

Then $z', z''$ are different from $x,y$.
By Lemma~\ref{obs_realization}, we may assume $x,y,z'', z,z'$ is the clockwise ordering on $C$.
See the last figure of Figure~\ref{fig:Case3}.
Let $p_1$ be the boundary neighbor of $y$ other than $x$. Let
$P$ be the clockwise subpath of $C$ joining $p_1$ and $z$.
Note that by our case assumption, $|V(P)|\ge 2$.
Let  $G'$ be the block of $G-V(P)$ containing $x, y,z'$, and let
$C'$ be the boundary  cycle of
$G'$.
Let $Q$ be the clockwise subpath of $C'$ joining $y$ and $z'$.

\begin{claim}\label{claim:Q}
Every two adjacent vertices $q$ and $q'$ on $Q$ have a common neighbor in $V(P)$.
\end{claim}
\begin{proof}
Since $G$ is a near triangulation, $q$ and $q'$  have a common neighbor $w$ in $V(G)\setminus V(G')$.
If $w$ is not on $P$, then $qq'$ cannot be a boundary edge of $G-V(P)$, which is a contradiction.
\end{proof}

Since there is no chord incident with $y$, $|V(Q)|\ge 3$.
By Claim~\ref{claim:Q}, every vertex of $Q$ has a neighbor in $V(P)$.
If every vertex of $Q$ has exactly one neighbor in $V(P)$, then 
$V(P)=\{z\}$, which is a contradiction.
Let $q_0=y$, $q_1,q_2,\ldots,q_k$ $(k\ge1)$ be the vertices of $Q$ in the order from $y$ to $z'$ that are adjacent to at least two vertices in $V(P)$ and let $q_{k+1}=z'$.
By Claim~\ref{claim:Q}, for $i\in\{1,2,\ldots,k+1\}$, let $p_{i} \in V(P)$ be the vertex adjacent to $q_{i-1}$ and $q_i$.
Note that $P$ is a path from $p_1$ to $p_{k+1}=z$, and $yp_1 \in E(G)$.

For $j\in\{0,1,\ldots,k\}$, let
$Q_j$ be the subpath of $Q$ from $q_{j}$ to $q_{j+1}$.
For $i\in\{1,\ldots,k\}$,  let $P_i$ be the subpath of $P$ from $p_i$ to $p_{i+1}$.
Let $C_i$ be the cycle consisting of $P_i$ and vertex $q_i$, and let
 $G_i$ be the maximal plane subgraph of $G$ with boundary cycle  $C_i$.
 Let $(D_i,H_i)$ be a $(2,6)$-decomposition of $G_i$ with respect to $(p_{i+1},q_{i},p_{i})$.
Then, clearly $(p_{i+1},q_i), (p_i,q_i) $ are arcs of  $D_i$.
Modify $D_i$ and $H_i$ by reversing the orientation of $(p_i,q_i)$ in $D_i$, removing $(p_{i+1},q_i)$ from $D_i$, and then adding $p_{i+1}q_i$ to $H_i$.
Then for $i\in\{1,\ldots,k\}$, $D_i$ is still acyclic and
\begin{eqnarray*}
\deg_{D_i}^+(q_i)= \deg_{H_i}(q_i)=1, \deg_{D_i}^+(p_i)=\deg_{D_i}^+(p_{i+1})=0, \deg_{H_i}(p_i)\le 3, \text{ and } \deg_{H_i}(p_{i+1})\le 2.
\end{eqnarray*}
Let $(D',H')$ be a $(2,6)$-decomposition of $G'$ with respect to $(x,y,z')$.
Let
\begin{eqnarray*}
D&=&D'\cup \left(\bigcup_{i=1}^k D_i \right) +
\{(p_1,y)\}+ \{ (q,p_{i+1}) \mid q \in V(Q_i)\setminus \{q_i,q_{i+1}\},  i\in\{0,1,\ldots,k\}\}, \\
H&=&H'\cup \left(\bigcup_{i=1}^k H_i\right)+zz'.
\end{eqnarray*}
Suppose $\deg_{H'}(x)=1$. Then the vertex $s$ such that $sx \in E(H')$ belongs to $V(Q)\setminus \{y,z'\}$ by the case assumption. Delete $sx$ from $H$ and then add arc $(s,x)$ to $D$.
For the smallest index $i$ such that $s\in V(Q_i)$,
if $i\ge 1$, then modify $D$ by reversing the orientation of $(s,p_{i+1})$ in $D$, and if $i=0$, then modify $D$ and $H$ by deleting arc $(s,p_{i+1})$ from $D$ and then adding the edge $sp_{i+1}$ to $H$.
Clearly, $D$ is acyclic.
From the definition of $(D,H)$,  $N_{D}^+(x)=\{y\}$ and $\deg_{D}^+(y) = \deg_{H}(x)=\deg_{H}(y) = 0$, so
Condition (iii) holds.

For the vertex $s$ (if it exists), $\deg^+_{D}(s)\le 2$ and $\deg_{H}(s)\le 6$. For   $w\in V(Q)\setminus\{y,z',s\}$,
\begin{eqnarray*}
&& \deg_{D}^+(w)=\deg_{D'}^+(w)+1\le 2 \quad \text{ and }\quad \deg_{H}(w)\le \deg_{H'}^+(w)+1\le 6,
\end{eqnarray*}
and therefore Condition (i) holds. It is easy to check that
\begin{eqnarray*}
&&\deg_{D}^+(p_i) \le 1, \text{ for }i\in\{1,2,\ldots,k\},\\
&& \deg_H(p_1)\le 1+3=5-b(p_1),\quad
\deg_H(p_i)\le 3+2 = 5-b(p_i)\text{ for }i\in\{2,3,\ldots,k\},\\
&&\deg_{D}^+(z) \leq 1,  \quad \deg_H(z) \le 1+\deg_{H_k}(z)\le 1+2 = 3 < 4 = 4-b(z)\\
&&\deg_{D}^+(z') \le 1,  \quad
\deg_{H}(z')=\deg_{H'}(z')+1 \le 5-b(z'), \text{ and }\\
&&\deg_{H}(z'')=
\deg_{H_k}(z'')\le 5-b_{G_k,p_{k+1}q_k}(z'')=4 \le 4-b(z'') \  \text{if }z''\neq p_{k}.
\end{eqnarray*}
If $z''=p_k$, then the boundary cycle of $G_k$ is a triangle.
Thus, $\deg_{H_k}(p_k)\le 4-b_{G_k,p_{k+1}q_{k}}(p_k)=2$, which implies that
\[\deg_H(z'')=\deg_H(p_k)\le \begin{cases}
2+2 =4=4-b(z'') &\text{if }z''=p_k, k>1,\\
1+2=3=4-b(z'') &\text{if }z''=p_1,
\end{cases}\]
so
Conditions (ii) and (iv) hold.

It remains to check Condition (v). As shown above, whether $z''$ is $p_k$
or not, we have  $\deg_{H}(z'') \le 4-b(z'')$. Therefore
$$\deg_{H}(z) + \deg_{H}(z') + \deg_{H}(z'')\le
3+(5-b(z'))+(4-b(z'')) = 12-b(z')-b(z'').$$
\end{proof}

We finish this section by proving the following, which implies Proposition~\ref{prop:(2,3)}.

\begin{proposition}
Let $G$ be a plane triangulation on at least $11$ vertices.
If $G'$ is the plane graph obtained from $G$ by adding a new vertex $v_f$ to every face $f$ of $G$ and adding all edges between $v_f$ and the vertices of $f$, then $G'$ is not  $(2,3)$-decomposable.
\end{proposition}

\begin{proof}
Let $n=|V(G)|$. Since $G$ is a triangulation, $G$ has $2n-4$ faces and $3n-6$ edges. Suppose to the contrary that $G'$ is $(2,3)$-decomposable.
Let $(D,H)$ be a $(2,3)$-decomposition of $G'$ that  maximizes $|E(H)\cap (E(G')\setminus E(G))|$.
Let $\sigma$ be a $2$-degenerate ordering of $D$.

We claim that for every face $f$ of $G$, $\deg_H(v_f)\ge 1$.
Suppose $\deg_H(v_f)=0$ for some face $f$ of $G$. Let $v_1,v_2,v_3$ be the vertices of $G$ incident with $f$.
Then $\deg_D(v_f)=3$, so some $v_{j}$ comes later than $v_f$ in $\sigma$.
We may assume $v_{1}$ is the last in $\sigma$ among $\{v_f, v_{1}, v_{2}, v_{3}\}$.
Since $\sigma$ is a  $2$-degenerate ordering, either $v_{1}v_{2}$ or $v_{1}v_{3}$ is in $H$, say $v_{1}v_{2} \in E(H)$.
Let $D'=D-v_fv_{1}+v_{1}v_{2}$ and $H'=H-v_{1}v_{2}+v_fv_{1}$.
Then $(D',H')$ is a $(2,3)$-decomposition of $G'$, which is a contradiction to the maximality of $|E(H)\cap (E(G')\setminus E(G))|$. Therefore $\deg_H(v_f)\ge 1$ for every face $f$ of $G$, and thus $|E(H)\setminus E(G)|\ge |F(G)|=2n-4$.

In $\sum_{v \in V(G)} \deg_H(v)$, an edge in $E(H)\cap E(G)$ is counted twice and an edge in $E(H)\setminus E(G)$ is counted once.
Hence, together with the fact that $\Delta(H)\le 3$,
\[ 3n\ge \sum_{v \in V(G)} \deg_H(v) \ge 2|E(H)\cap E(G)|+|E(H)\setminus E(G)|\ge 2|E(H)\cap E(G)|+(2n-4).\]
From the fact that $|E(D)\cap E(G)|\le 2n-3$, we have
$|E(H)\cap E(G)|\ge (3n-6)-(2n-3)=n-3$, so
$3n \ge 2(n-3)+2n-4=4n-10$,
which is a contradiction since $n\ge 11$.
\end{proof}

\section{Proof of  (3,2)-decomposability}\label{sec:3Dplus2}

Note that for a near triangulation  and a boundary edge $xy$, there always exists a boundary vertex $z$ distinct from $x$, $y$ that is not incident with a chord of the boundary cycle.
Instead of proving   Theorem~\ref{thm:3Dplus2} directly, we prove the following more technical result.

\begin{theorem}\label{thm:32}
Let $G$ be a near triangulation, $xy$ be a boundary edge of $G$, and $z$ be a boundary vertex other than $x$, $y$ that is not incident with a chord of the boundary cycle.
When neither $x$ nor $y$ is a boundary neighbor of $z$, let $z'$ be a boundary neighbor of $z$.
Then there exist a subgraph $H$ and an acyclic orientation $D$ of $G-E(H)$ satisfying the following:
\begin{itemize}
\item[\rm{(i)}]  For every interior vertex $w$, $\deg_D^{+}(w)\le 3$ and $\deg_H(w)\le 2$.
\item[\rm{(ii)}] For every boundary vertex $w$, $\deg_D^{+}(w)\le 2$ and $\deg_H(w)\le 2$. Moreover, if $w \ne z'$,
then $\deg_D^{+}(w)+\deg_H(w) \le 3$.
\item[\rm{(iii)}] $\deg_D^+(y)=\deg_H(x)=\deg_H(y)=0$, $N_D^+(x)=\{y\}$, and $\deg_D^{+}(z)+\deg_H(z) \le 2$.
\end{itemize}
Let us call such $(D,H)$ a $(3,2)$-decomposition of $G$ with respect to $(x,y,z)$ or $(x,y,z,z')$.
\end{theorem}

\begin{proof}
We use induction on $|V(G)|$. 
If $|V(G)|= 3$, then $G=K_3$.
Let $D$ be a digraph with arcs $(x,y)$, $(z,x)$ and  $(z,y)$,  and $H$ be the empty graph.
Then $(D,H)$ is a $(3,2)$-decomposition of $G$ with respect to $(x,y,z)$. Suppose $|V(G)|\ge 4$.
Let $C$ be the boundary cycle of $G$.

\medskip

\noindent {\bf Case 1}   $C =(x,y,z)$ is a triangle.

Let $G'=G-z$.
Note that $G'$ is a near triangulation and  let $C'$ be  the boundary cycle of $G'$. Let $w\in N_G(z) \setminus \{x,y\}$ such that $w$ is not incident with a chord of $C'$. By the induction hypothesis, there is a $(3,2)$-decomposition $(D',H')$ of $G'$ with respect to $(x,y,w,w')$ or $(x,y,w)$ depending on the existence of $w'$.
Then $(D,H)$, where $D=D'+\{(z,y),(z,x)\}+\{ (u,z)\mid u\in V(C')\setminus\{x,y\} \}$ and $H=H'$, satisfies Conditions (i)-(iii).

\medskip

\noindent {\bf Case 2}  $C$ has a chord $uv$.

\begin{figure}[h!]
\centering   \includegraphics[width=4cm,page=9]{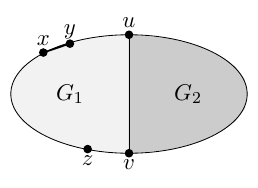}
  \qquad \includegraphics[width=4cm,page=10]{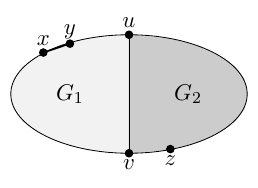}\\
   \caption{Illustrations for \textbf{Case 2}}
  \label{fig:(3,2)_triangle}
\end{figure}

\noindent {\bf Case 2-1} There is a chord $uv$ of $C$ such that $x,y,z \in V(G_i)$ for some $i \in \{1,2\}$, where $G_1$ and $G_2$ are the plane subgraphs of $G$ separated by $uv$.

Let $C_i$ be the boundary cycle of $G_i$. Without loss of generality, assume $x,y,z \in V(G_1)$.
See the first figure of  Figure~\ref{fig:(3,2)_triangle}. 
Choose the chord $uv$ so that $G_2$ is minimum, so $C_2$ has no chord.
Note that $z \not\in\{u, v\}$,  since $z$ is not incident with a chord of $C$.
Therefore, $z' \in V(G_1)$ if neither $x$ nor $y$ is a boundary neighbor of $z$ in $G$.
By the induction hypothesis, there is a $(3,2)$-decomposition $(D_1,H_1)$ of $G_1$ with respect to $(x,y,z,z')$ or $(x,y,z)$ depending on the existence of  $z'$.
Let $z''$ be a boundary neighbor of $v$ in $G_2$ other than $u$.
By the induction hypothesis, there is
 a $(3,2)$-decomposition $(D_2,H_2)$ of $G_2$ with respect to $(u,v,z'')$.
Note that $z''$ is not incident with a chord of $C_2$ since it has no chord.
Let $D=D_1+(D_2-uv)$ and $H=H_1+(H_2-uv)$.
Since $N_{D_2}^+(u)=\{v\}$, $\deg_{D_2}^+(v)=\deg_{H_2}(u)=\deg_{H_2}(v)=0$ by Condition (iii) for $(D_2,H_2)$, it follows that $D$ is acyclic and Conditions (i)-(iii) are  easily verified.

\smallskip

\noindent {\bf Case 2-2} For every chord $uv$ of $C$, $x,y \in V(G_1)$ and $z \in V(G_2)\setminus V(G_1)$,
where $G_1$ and $G_2$ are the plane subgraphs of $G$ separated by $uv$. See the second figure of  Figure~\ref{fig:(3,2)_triangle}.

Let $C_i$ be the boundary cycle of $G_i$. 
Choose the chord $uv$ so that $G_1$ is minimum, so $C_1$ has no chord.
By the induction hypothesis, there is a $(3,2)$-decomposition $(D_1,H_1)$ of $G_1$ with respect to $(x,y,w)$  where $w$ is a boundary vertex of $G_1$ so that $wx$ is a boundary edge of $G_1$.
Note that $w$ is not incident with a chord of $C_1$.

If either $zu$ or $zv$ is a boundary edge of $G$, then there exists a $(3,2)$-decomposition $(D_2,H_2)$ of $G_2$ with respect to $(u,v,z)$ by the induction hypothesis.
If $z$ is neither adjacent to $u$ nor $v$, then $z' \in V(G_2) \setminus \{u,v\}$, so let $(D_2, H_2)$ be a $(3,2)$-decomposition  of $G_2$ with respect to $(u,v,z,z')$.
Let  $D=D_1+(D_2-uv)$ and $H=H_1+(H_2-uv)$.
Since $N_{D_2}^+(u)=\{v\}$, $\deg_{D_2}^+(v)=\deg_{H_2}(u)=\deg_{H_2}(v)=0$ by Condition (iii) for $(D_2,H_2)$, it follows that $D$ is acyclic and Conditions (i)-(iii) are also easily verified.

\medskip

\noindent {\bf Case 3}   $C$ is not a triangle and has no chord.

\begin{figure}[h!]
\centering
  \includegraphics[height=2.8cm,page=11]{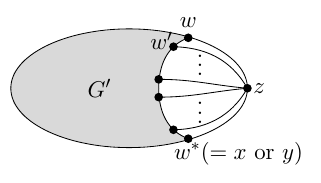}
  \includegraphics[height=2.8cm,page=12]{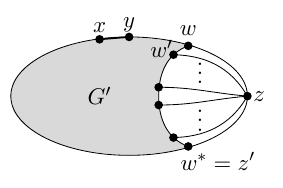}
  \quad
    \includegraphics[height=2.8cm,page=13]{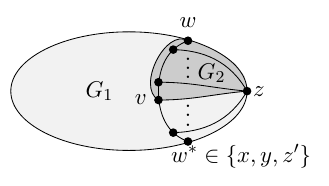}
   \caption{Illustrations for \textbf{Case 3}} \label{fig:nochord_triangle}
\end{figure}

Let $zw$ be the boundary edge of $G$ where $w\not\in\{x,y,z'\}$, and let $w^*$ be the other boundary neighbor of $z$ in $G$. Note that $w^*\in \{x,y,z'\}$.
For simplicity, let $U=N_G(z)\setminus\{w,w^*\}$.
Let $G'=G-z$.
Note that $G'$ is a near triangulation, and let $C'$ be the boundary cycle of $G'$.
Let $w'$ be the interior vertex of $G$ which is a boundary neighbor of $w$ in  $G'$.
(Such $w'$ exists,  since  $G$ has no chord and so $\deg_G(z)\ge 3$.)

\smallskip

\noindent {\bf Case 3-1} $C'$ has no chord at the vertex $w$.

We find a (3,2)-decomposition $(D',H')$ of $G'$ with respect to $(x,y,w,w')$
(if $y$ or $x$ is a boundary neighbor of $w$ in $G$, then we do not consider $w'$) by the induction hypothesis.
Note that $\deg^+_{D'}(w)+\deg_{H'}(w)\le 2$.
Let $\tilde{D}=D'+\{ (u,z) \mid u\in U\}$ for simplicity.

Suppose $w^*\in\{x, y\}$.
See the first figure of Figure \ref{fig:nochord_triangle}.
If $\deg_{H'}(w)\le 1$, then let  $D=\tilde{D}+(z,w^*)$ and $H=H'+zw$.
If $\deg_{H'}(w)= 2$, then let $D=\tilde{D}+\{(w,z),(z,w^*)\}$ and $H=H'$.
Since $\deg_{D'}^+(y)=0$ and $N^+_{D'}(x)=\{y\}$,
it follows that $D$ is acyclic.
Moreover, $\deg_{H'}(w)= 2$ implies
$\deg_{D'}^+(w)=0$, so Conditions (i)-(iii) are verified.

Suppose $w^*=z'$.
See the second figure of Figure \ref{fig:nochord_triangle}.
We divide into four cases according to $\deg_{H'}(w)$ and
$\deg_{D'}^+(z')$.
Note that $\deg_{H'}(w)=2$  implies $\deg_{D'}^+(w)=0$.
\begin{itemize}
\item If $\deg_{H'}(w)\le 1$ and $\deg_{D'}^+(z') \le 1$, then
let $D=\tilde{D}+\{(z',z)\}$ and $H=H'+zw$.
\item If $\deg_{H'}(w)=2$ and $\deg_{D'}^+(z') \le 1$, then
let $D=\tilde{D}+\{(z',z),(w,z)\}$ and $H=H'$.
\item If $\deg_{H'}(w)\le 1$ and $\deg_{D'}^+(z')=2$, then
let $D=\tilde{D}$ and $H=H'+\{zw,zz'\}$.
\item If $\deg_{H'}(w)=2$ and $\deg_{D'}^+(z') =2$, then
let $D=\tilde{D}+\{(w,z)\}$ and $H=H'+zz'$.
\end{itemize}
Clearly, the resulting digraph $D$ is acyclic.
It is also easy to check Conditions (i) and (iii).
By Condition (ii) for $(D',H')$, we have $\deg_{D'}^+(z')+\deg_{H'}(z')\le 3$, so
Condition (ii)  is also satisfied.

\smallskip

\noindent {\bf Case 3-2} $C'$ has a chord $wv$.

Since there is no chord of $C$ by the case assumption,
$v \in N_G(z) \setminus \{w^*,w'\}$.
Then $G-\{z,w,v\}$ has two components $V_1$ and $V_2$.
Let $G_i =G[V_i \cup \{z,w,v\}]$ for each $i$, and assume $x,y \in V(G_1)$.
Note that each $G_i$ is a near triangulation.
See the last figure of Figure~\ref{fig:nochord_triangle}.
By the induction hypothesis, there is a $(3,2)$-decomposition $(D_1,H_1)$ of $G_1$ with respect to $(x,y,z,z')$ (if either $zy$ or $zx$ is a boundary edge of $G_1$ (or $G$), then we do not consider $z'$).
By the induction hypothesis, there is a $(3,2)$-decomposition $(D_2,H_2)$ of $G_2$ with respect to $(w,v,z)$.
Let $D=D_1+(D_2-\{zw,vz,vw\})$ and $H=H_1+H_2$.

By Condition (iii) for $(D_2,H_2)$,  $(s,t)$ is an arc of $D_2$, for every edge $st$ of $G$ joining an interior vertex $s$ of $G_2$ and a boundary vertex $t$ of $G_2$.
Hence, $D$ is acyclic and Conditions (i)-(iii) are easily verified.
\end{proof}

\section{Proof of $(4, 1)$-decomposability}
\label{sec:4DplusM}

A {\it $d$-vertex}, a {\it $d^+$-vertex}, and a {\it $d^-$-vertex} are a vertex of degree $d$, at least $d$, and at most $d$, respectively.
A {\it $d$-neighbor} is a neighbor that is a $d$-vertex.
A {\it $d^+$-neighbor} and a {\it $d^-$-neighbor} are defined analogously.
Note that even though a matching is a collection of edges, we sometimes refer to it as a subgraph with maximum degree one.

Let $G$ be a minimum counterexample to Theorem~\ref{thm:4DplusM} with respect to the number of vertices.
We may assume that $G$ is a triangulation, and fix an embedding of $G$.
The following lemma reveals some reducible configurations of $G$.

\begin{lemma}\label{Lemma:M+4D}
The following structures cannot appear in $G$:
\begin{itemize}
\item[\rm(i)] A $4^-$-vertex.
\item[\rm(ii)] Two adjacent $5$-vertices.
\item[\rm(iii)] A $5$-vertex with three consecutive $6^-$-neighbors.
\item[\rm(iv)] A $5$-vertex with two $7$-neighbors and  three $6^-$-neighbors.
\item[\rm(v)] A $7$-vertex with three consecutive $6^-$-neighbors where two of them are $5$-vertices.
\end{itemize}
\end{lemma}
\begin{proof}
In all cases, we will obtain a $(4,1)$-decomposition of $G$, which is a contradiction.

(i) Suppose to the contrary that there is a $4^-$-vertex $v$.
By the minimality of $G$, $G-v$ has a $(4,1)$-decomposition $(D',M')$  with a $4$-degenerate ordering $\sigma'$ of $D'$.
Let $M=M'$, and let $D$ be the graph from $D'$ by adding all edges incident to $v$.
Clearly, $M$ is a matching and the ordering obtained by appending $v$ to $\sigma'$ is a $4$-degenerate ordering of $D$, so $D$ is $4$-degenerate.

(ii) Suppose to the contrary that there are two adjacent $5$-vertices $u$ and $v$. 
By the minimality of $G$, $G-\{u,v\}$ has a $(4,1)$-decomposition $(D',M')$
with a $4$-degenerate ordering $\sigma'$ of $D'$.
Let $M=M'\cup\{uv\}$ and let $D=G-M$. Clearly, $M$ is a matching and the ordering obtained by appending
$v,u$ to $\sigma'$ is a $4$-degenerate ordering of $D$, so $D$ is $4$-degenerate.

(iii) Suppose to the contrary that there is a $5$-vertex $v$ with three $6^-$-neighbors $u_1$, $u_2$, $u_3$, and  $u_1u_2, u_2u_3\in E(G)$.
By the minimality of $G$, $G-\{v,u_1,u_2,u_3\}$ has a $(4,1)$-decomposition $(D',M')$ with a
$4$-degenerate ordering $\sigma'$ of $D'$.
Let $M=M'\cup \{vu_1,u_2u_3\}$ and let $D=G-M$.
Clearly, $M$ is a matching, and the ordering obtained by appending $u_3,u_1,u_2,v$ to $\sigma'$ is a $4$-degenerate ordering of $D$, so $D$ is $4$-degenerate.

(iv) Suppose to the contrary that there is a $5$-vertex $v$ with three $6^-$-neighbors and two $7$-neighbors.
Let $N_{G}(v)=\{u_1,u_2,u_3,u_4,u_5\}$ where $u_1u_5\in E(G)$ and $u_iu_{i+1}\in E(G)$ for $i\in\{1,2,3,4\}$.

By (ii) and (iii), we may assume that $u_1$, $u_2$, $u_4$ are the $6$-vertices, and $u_3$ and $u_5$ are the $7$-vertices.
By the minimality of $G$, $G-N_G[v]$ has a $(4,1)$-decomposition $(D',M')$ with
a $4$-degenerate ordering $\sigma'$ of $D'$.
Let $M=M'\cup \{vu_5,u_1u_2,u_3u_4\}$ and let $D=G-M$.
Clearly, $M$ is a matching, and the ordering obtained by appending  $u_5,u_1,u_3,u_2,u_4,v$ to $\sigma'$ is a $4$-degenerate ordering of $D$, so $D$ is $4$-degenerate.

(v) Suppose to the contrary that there is a $7$-vertex $v$ with three consecutive neighbors $u_1,u_2,u_3$ where two of them are $5$-vertices. By (ii), $u_1$, $u_3$ are $5$-vertices and $u_2$ is a $6$-vertex.
By the minimality of $G$, $G-\{v,u_1,u_2,u_3\}$ has a $(4,1)$-decomposition $(D',M')$
with a $4$-degenerate ordering $\sigma'$ of $D'$.
Let $M=M'\cup \{vu_1,  u_2u_3\}$ and let $D=G-M$.
Clearly, $M$ is a matching, and the
 ordering obtained by appending  $v,u_2,u_3,u_1$ to $\sigma'$ is a $4$-degenerate ordering of $D$, so $D$ is $4$-degenerate.
\end{proof}

We use the discharging method to reach the final contradiction, to conclude that the minimum counterexample $G$ could not have existed.
By Euler's formula, recall that \[\sum_{v\in V(G)} (\deg_G(v)-6)+\sum_{f\in F(G)}(2\deg_G(f)-6)=-12.\]
Since $\deg_G(f)\ge 3$ for every face $f\in F(G)$, we know \[\sum_{v\in V(G)} (\deg_G(v)-6)\le -12.\]
Let the initial charge of each vertex $v$ be $\deg_G(v)-6$, and note that the initial charge sum is negative.
We will reach a contradiction by showing that the final charge at each vertex is non-negative after the discharging rules, which preserves the charge sum.
The following is our one discharging rule:
\begin{itemize}
    \item[\textbf{[R]}] Each $6^+$-vertex $v$ sends charge $(\deg_G(v)-6)/d_5(v)$ to each of its $5$-neighbors, where $d_5(v)$ is the number of $5$-neighbors of $v$.
\end{itemize}

By Lemma~\ref{Lemma:M+4D}~(ii),    $d_5(v) \le \left\lfloor\frac{\deg_G(v)}{2}\right\rfloor$.
Thus, an $8$-vertex and a $7$-vertex send charge at least $\frac{1}{2}$ and at least $\frac{1}{3}$, respectively, to each $5$-neighbor.

By the rule [\textbf{R}], the final charge of a $6^+$-vertex is non-negative.
By Lemma~\ref{Lemma:M+4D} (i), it remains to check $5$-vertices.
Take a $5$-vertex $v$, and let $N_{G}(v)=\{u_1,u_2,u_3,u_4,u_5\}$ where $u_1u_5\in E(G)$ and $u_iu_{i+1}\in E(G)$ for $i\in\{1,2,3,4\}$.
If $v$ has at least two $8^+$-neighbors, then the final charge of $v$ is non-negative.
If $v$ has no $8^+$-neighbors, then by Lemma~\ref{Lemma:M+4D}~(iii) and (iv), it has at least three $7$-neighbors, and the final charge of $v$ is non-negative.

Assume $v$ has exactly one $8^+$-neighbor $u_5$.
If $v$ has at least two $7^+$-neighbors other than $u_5$, then it has non-negative final charge.
If $v$ has no $7^+$-neighbor other than $u_5$, then this is a contradiction to Lemma~\ref{Lemma:M+4D} (iii).
Thus, $v$ has exactly one $7$-neighbor, so it has three $6^-$-neighbors.
By Lemma~\ref{Lemma:M+4D}~(iii), we may assume that $u_3$ is the $7$-neighbor and $u_1,u_2,u_4$ are the $6^-$-neighbors.
By Lemma~\ref{Lemma:M+4D}~(ii) and (v), the $7$-vertex $u_3$ has at most two $5$-neighbors.
Thus $u_3$ sends charge at least $\frac{1}{2}$ to $v$ by the rule [\textbf{R}].
Since $u_5$ sends charge at least $\frac{1}{2}$ to $v$ by the rule [\textbf{R}], the  final charge of $v$ is non-negative.

\section*{Acknowledgements}
\small{This work has started during the 5th Korean Early Career Researcher Workshop in Combinatorics.}

\small{Ilkyoo Choi was supported by the Basic Science Research Program through the National Research Foundation of Korea funded by the Ministry of Education (No. NRF-2018R1D1A1B07043049), and also by the Hankuk University of Foreign Studies Research Fund.
Ringi Kim was supported by the National Research Foundation of Korea grant funded by the Korea government (No. NRF-2018R1C1B6003786), and also by Basic Science Research Program through the National Research Foundation of Korea funded by the Ministry of Education (No. NRF-2019R1A6A1A10073887).
Boram Park was supported by the National Research Foundation of Korea grant funded by the Korea government (No. NRF-2018R1C1B6003577).
Xuding Zhu was supported by NSFC 11971438 and 111 project of Ministry of Education of China.}

\bibliography{ref}{}
\bibliographystyle{plain}

\end{document}